\documentclass[a4paper,12pt]{article}
\usepackage{natbib}
\usepackage{latexsym}
\usepackage{amssymb,amsthm,amsmath}
\usepackage{hyperref}
\usepackage{amsfonts}
\DeclareMathSymbol \tick {\mathord}{AMSa}{"58}
\makeatletter
\newcommand* \cross {\mathpalette \cr@ss {}}
\newcommand* \cr@ss@ {.0175\wd\z@}
\newcommand* \cr@ss [1]{{%
        \setbox\z@ = \hbox {$#1\times$}%
        \copy\z@
        \kern-1\cr@ss@
        \copy\z@
        \kern-.965\wd\z@
        \copy\z@
        \kern-1\cr@ss@
        \lower \cr@ss@ \copy\z@
        \kern-\wd\z@
        \raise \cr@ss@ \box\z@
}}
\makeatletter

\newtheorem{theorem}{Theorem}
\newtheorem{corollary}[theorem]{Corollary}
\newtheorem{lemma}[theorem]{Lemma}
\newtheorem{proposition}[theorem]{Proposition}
\theoremstyle{definition}
\newtheorem{definition}[theorem]{Definition}
\theoremstyle{remark}
\newtheorem{example}[theorem]{Example}
\newtheorem{remark}[theorem]{Remark}
\newcommand* \ba {\mathop{\rm ba}\nolimits }
\newcommand* \ca {\mathop{\rm ca}\nolimits }
\newcommand* \mycore [1]{\mathop{\rm #1\mathchar`-core}\nolimits }
\newcommand* \bacore {\mycore {ba}}
\newcommand* \cacore {\mycore {ca}}
\newcommand*  \kcore {\mycore \kappa}
\makeatletter
\newcommand* \eqarr@ywithdef [1]{%
	\null\,%
	\vcenter {%
		\normalbaselines
		\openup\jot
		\m@th
		\ialign {\strut #1\crcr }%
	}\,%
}
\newcommand* \eqarraywithdef [1]{%
	\eqarr@ywithdef {\ignorespaces #1}%
}
\makeatother

\title{The $\kappa$-core and the $\kappa$-balancedness of TU games\thanks{David Bartl acknowledges the support of the Czech Science Foundation
under grant number GA\v CR 21-03085S.  A part of this research was done
while he was visiting the Corvinus Institute for Advanced Studies; the
support of the CIAS during the stay is gratefully acknowledged. Mikl\'os Pint\'er acknowledges the support by the Hungarian Scientific Research Fund under projects K 133882 and K 119930.}}

\author{David Bartl\thanks{Department of Informatics and Mathematics, School of Business Administration in Karvin\'a, Silesian University in Opava, bartl@opf.slu.cz.} and Mikl\'os Pint\'er\thanks{Corresponding author: Corvinus Center for Operations Research, Corvinus Institute of Advanced Studies, Corvinus University of Budapest, pmiklos@protonmail.com.}}

\begin{document}

\maketitle

\begin{abstract}
We consider transferable utility cooperative games with infinitely many
players. In particular, we generalize the notions of core and
balancedness, and also the Bondareva-Shapley Theorem for infinite
TU-games with and without restricted cooperation, to the cases where the
core consists of $\kappa$-additive set functions.
% Transferable utility cooperative games with infinitely many players
% are considered.  We generalize the notions of core and balancedness to
% the cases where the core consists of $\kappa$-additive set functions.
% Furthermore, we generlaize also the Bondareva-Shapley Theorem for
% infinite TU games with and without restricted cooperation.
Our generalized Bondareva-Shapley Theorem extends previous results by \citet{Bondareva1963}, \citet{Shapley1967}, \citet{Schmeidler1967}, \citet{Faigle1989}, \citet{Kannai1969,Kannai1992}, \citet{Pinter2011d} and \citet{BartlPinter2022a}.


%\textit{JEL Classification:} C71
\end{abstract}

\section{Introduction}

%SSK paper !!!! $\kappa$-additivity

The core \citep{Shapley1955,Gillies1959} is definitely one of the most important solution concepts of cooperative game theory. In the transferable utility setting (henceforth games) the Bondareva-Shapley Theorem \citep{Bondareva1963,Shapley1967,Faigle1989} provides a necessary and sufficient condition for the non-emptiness of the core; it states that the core of a game with or without restricted cooperation is not empty if and only if the game is balanced. The textbook proof of the Bondareva-Shapley Theorem goes by the strong duality theorem for linear programs (henceforth LPs), see e.g.\ \citet{PelegSudholter2007}. The primal problem corresponds to the concept of balancedness and so does the dual problem to the notion of core. However, this result is formalized for games with finitely many players. It is a question how one can generalize this result to the infinitely many player case.

%In each of the three cases we consider games without and with restricted cooperation. In the finite case Corollary~\ref{cor1} says that our $\kappa$-balancedness is an extension of the notion of balancedness known for finite games without restricted cooperation \citep{Bondareva1963,Shapley1967} as well as with restricted cooperation \citep{Faigle1989}. In other words, $\kappa$-balancedness and the ``finite game balancedness" are equivalent for finite games without and with restricted cooperation.

The finitely many player case is special in (at least) two counts: (1)~it can be handled by finite linear programs, (2)~since the power set of the player set is also finite, it is natural to take the solution of a game from the set of additive set functions (additive games).

There are two main directions to generalize the notion of additive set function. The first, when we weaken the notion of additivity; this leads to the notion of $k$-additive core \citep{GrabischMiranda2008}, where $k$~is a finite cardinal (natural number). The second, when we use a notion stronger than additivity (e.g.\ $\sigma$-additivity). Naturally, a stronger notion does matter only if there are infinitely many players. This latter approach is considered here.

\citet{Schmeidler1967}, \citet{Kannai1969,Kannai1992}, \citet{Pinter2011d}, and \citet{BartlPinter2022a} considered games with infinitely many players. All these papers studied the additive core; that is, the case when the core consists of bounded additive set functions. \citet{Schmeidler1967} and \citet{Kannai1969} showed that the additive core of a non-negative game without restricted cooperation with infinitely many players is not empty if and only if the game is Schmeidler balanced (Definition~\ref{kiegyensulyozott3}). \citet{BartlPinter2022a} extended these results and showed that the additive core of a game bounded below with our without restricted cooperation with infinitely many players is not empty if and only if the game is (bounded-)Schmeidler balanced.
%We extend this result and show that the additive core of a bounded below TU game without restricted cooperation is not empty if and only if the TU game is Schmeidler balanced (Theorem~\ref{megszamjatek-Schmeidler}).

\citet{Kannai1992} raised the following two research questions: (1)~When does there exist a bounded $\sigma$-additive set function in the core? (2)~When are all elements in the core bounded $\sigma$-additive?

\citet{Kannai1969} gave a necessary and sufficient condition for that the $\sigma$-additive core of a non-negative game without restricted cooperation and with infinitely many players is not empty; that is, he answered Question~(1). This result (the necessary and sufficient condition) is, however, only slightly similar to the classical balancedness condition. Moreover, it works only for non-negative games without restricted cooperation. \citet{Schmeidler1972} and \citet{EinyHolzmanMondererShitovitz1997} answered Question~(2) respectively for exact and for continuous convex games.

In this paper we raise the following question, where $\kappa$~is an infinite cardinal number, and generalize \citeauthor{Kannai1992}'s first question thus: When does there exist a bounded $\kappa$-additive set function in the core? Moreover, we consider this question in the case of games with restricted cooperation too.

Addressing this question, we introduce the notions of $\kappa$-core and $\kappa$-balan\-cedness (Definitions \ref{core} and~\ref{kiegyensulyozott}). Then, we apply the strong duality theorem for infinite LPs by \citet{AndersonNash1987} (Proposition~\ref{duality}) and prove that the $\kappa$-core of a game with or without restricted cooperation and with arbitrarily many players is not empty if and only if the game is $\kappa$-balanced (Theorem~\ref{megszamjatek1}).

%Regarding further results, we extend \citeauthor{Schmeidler1967}'s \citeyearpar{Schmeidler1967} result and show that  the additive core of a bounded below TU game without restricted cooperation is not empty if and only if the TU game is Schmeidler balanced. We also compare $\kappa$-balancedness  (Definition \ref{kiegyensulyozott}) and \citeauthor{Kannai1992}'s \citeyearpar{Kannai1992} above mentioned necessary and sufficient condition.

%Regarding applications, the non-emptiness of the core is also relevant from the viewpoint of decision theory, where a non-additive belief is a monotone TU game, and the so called priors (possible distributions, beliefs) are the elements of the core of the corresponding TU game \citep{Schmeidler1989,GilboaSchmeidler1989,GhirardatoMarinacci2002}. Recently \citet{Cerreia-VioglioGhirardatoMaccheroniMarinacciSiniscalchi2011} showed that under quite general assumptions a non-additive belief expresses ambiguity aversion if and only if its core is not empty. In other words, a necessary and sufficient condition for the non-emptiness of the core is a necessary and sufficient condition of that a non-additive belief expresses ambiguity aversion.

The set-up of the paper is as follows. In the next section we introduce the main mathematical notions and results, which are related to infinite LPs, and used in this paper. In Section~\ref{sec2} we introduce the notion of $\kappa$-additive set functions and discuss some related concepts and results.
In Section~\ref{main} we present game theory notions and define various
cores (such as $\kappa$-core) and balancedness conditions (such as
$\kappa$-balancedness) we consider in this paper.
% In Section~\ref{main} we give the game theory notions and the various
% core (such as $\kappa$-core) and balancedness (such as
% $\kappa$-balancedness) notions we consider in this paper. 
Section~\ref{main2} presents our main result. We give an answer to the question we have raised: there is a bounded $\kappa$-additive set function in the core if and only if the game is $\kappa$-balanced (Theorem~\ref{megszamjatek1}). The last section briefly concludes.

\section{Duality theorem}

In this section we discuss the duality theorem for infinite linear programs that we will use later.

Let $X$~and~$Y$ be real vector spaces; the {\it algebraic dual\/} of~$X$, which is the space of all linear functionals on~$X$, is denoted by~$X'$; similarly $Y'$ denotes the algebraic dual of~$Y$. Moreover, $Y^* \subseteq Y'$ denotes a linear subspace of~$Y'$ such that $(Y, Y^*)$ is a {\it dual pair of spaces}; that is, if $f \in Y$ is non-zero, then there exists a $y \in Y^*$ such that $y(f) \neq 0$. For any linear mapping $A \colon X \to Y$ its {\it adjoint mapping\/} is $A' \colon Y' \to X'$ with $\bigl(A'(y)\bigr)(x) = y\bigl(A(x)\bigr)$ for all $x \in X$ and $y \in Y'$. Moreover, a subset $P \subseteq X$ of the vector space~$X$ is a {\it convex cone\/} if $\alpha x + \beta y \in P$ for all $x,y \in P$ and all non-negative $\alpha,\beta \in \mathbb{R}$. For any two functionals $f,g \colon X \to \mathbb{R}$ we write $f \geq_P g$ if $f(x) \geq g(x)$ for all $x \in P$.

Now, given a linear mapping $A \colon X \to Y$, a point $b \in Y$ and a linear functional $c \colon X \to \mathbb{R}$, let us consider the following infinite LP-pair \citep[cf.\spacefactor=1000 ][Section~3.3]{AndersonNash1987}:
\begin{equation}\label{LP}
\eqarraywithdef {
\hfil $#$& ${}#$\hfil & \hfil ${}#{}$\hfil & $#$\hfil & \qquad\qquad
\hfil $#$& ${}#$\hfil & \hfil ${}#{}$\hfil & $#$\hfil \cr
(\mathrm{P_{LP}}) \qquad\qquad c(x&) & \multispan2${} \to \sup$\hfil &
(\mathrm{D_{LP}}) \qquad\qquad y(b&) & \multispan2${} \to \inf$\hfil \cr
\noalign {\kern\jot}
\hbox {s.t.} \enskip A  (x&) & =      & \hphantom{(} b &
\hbox {s.t.} \enskip A' (y&) & \geq_P &              c \cr
x&& \in & \hphantom{(} P   &
y&& \in &              Y^* }
\end{equation}
where $P \subseteq X$ is a convex cone and $Y^*$~is a subspace of~$Y'$ such that $(Y, Y^*)$ is a dual pair of spaces.

\begin{definition}
The program $(\mathrm{D_{LP}})$~is {\it consistent\/} if there exists a linear functional $y \in Y^*$ such that $\bigl(A' (y)\bigr)(x) \geq c(x)$ for all $x \in P$. The {\it value\/} of a consistent program~$(\mathrm{D_{LP}})$ is $\inf \bigl\{\, y(b) : A'(y) \geq_P c, \allowbreak\!\; y \in Y^* \,\bigr\}$.
\end{definition}

In the next definition we assume the weak topology on the space~$Y$ with respect to~$Y^*$. To introduce that, we describe all the neighborhoods of a point. A set $U \subseteq Y$ is a {\it weak neighborhood\/} of a point $f_0 \in Y$ if there exist a natural number~$n$ and functionals $y_1, \ldots, \allowbreak y_n \in Y^*$ such that $\bigcap_{j=1}^n \bigl\{\, f \in Y : \bigl|y_j(f) - y_j(f_0)\bigr| < 1 \,\bigr\} \subseteq U$.

\begin{definition}\label{supervalue}
Put $D = \bigl\{\, \bigl(A(x), \, c(x) \bigr) : x \in P \,\bigr\}$. The program $(\mathrm{P_{LP}})$~is {\it superconsistent\/} if there exists a $z \in \mathbb{R}$ such that $(b,z) \in \,\overline{\!D}$, where $\overline{\!D}$~is the closure of~$D$. The {\it supervalue\/} of a superconsistent program~$(\mathrm{P_{LP}})$ is $\sup \bigl\{\, z : (b,z) \in \,\overline{\!D} \,\bigr\}$.
\end{definition}

We recall that a pair $(I,\leq)$ is {\it right-directed\/} if $I$~is a preordered set and for any $i,j \in I$ there exists a $k \in I$ such that $i \leq k$ and $j \leq k$. A {\it net\/} (generalized sequence) of~$X$ is $(x_i)_{i\in I}$ where $(I,\leq)$ is a right-directed pair and $x_i \in X$ for all $i \in I$.

Notice that the program $(\mathrm{P_{LP}})$~is superconsistent if there exists a net $(x_i)_{i\in I}$ from~$P$ such that $A(x_i) \buildrel w \over \longrightarrow b$, which means that $A(x_i)$ converges to~$b$ in the weak topology, and $(c(x_i))_{i \in I}$ is bounded. Furthermore, a number $z^\ast$~is the supervalue of a superconsistent program~$(\mathrm{P_{LP}})$ if it is the least upper bound of all numbers $z$~such that there exists a net $(x_i)_{i\in I}$ from~$P$ such that $A(x_i) \buildrel w \over \longrightarrow b$ and $c(x_i) \longrightarrow z$.

\begin{proposition}\label{duality}
Consider the programs in~\eqref{LP}. Program $(\mathrm{P_{LP}})$~is superconsistent and $z^\ast$~is its finite supervalue if and only if program $(\mathrm{D_{LP}})$~is consistent and $z^\ast$~is its finite value.
\end{proposition}

Proposition~\ref{duality} is a restatement of Theorem~3.3, p.~41, in
\citet{AndersonNash1987}. Notice that we differ from
\citet{AndersonNash1987} in the point that \citeauthor{AndersonNash1987}
use slightly different notions of superconsistency and supervalue.
However, they also remark that their notions and the ones we use here
are equivalent (p.~41 above Theorem~3.3). This is why we omit the proof
of Proposition~\ref{duality} here.

\section{The $\kappa$-structures}\label{sec2}

Throughout this section $\kappa$ is an infinite cardinal number. Let $N$~be a non-empty set and let $\mathcal{A} \subseteq
\mathcal{P}(N)$ be a field of sets; that is, if $S_1, \ldots,
\allowbreak S_n \in \mathcal{A}$, then $\bigcup_{j=1}^n S_j \in
\mathcal{A}$, and $N \in \mathcal{A}$ with $N \setminus S \in
\mathcal{A}$ for any $S \in \mathcal{A}$. The pair $(N,\mathcal{A})$ is called {\it chargeable space}.

Given a chargeable space $(N,\mathcal{A})$, let $\ba (\mathcal{A})$ and
$\ca (\mathcal{A})$ denote, respectively, the set of bounded additive
set functions and the set of bounded $\sigma$-additive set functions
$\mu\colon \mathcal{A} \to \mathbb{R}$.

%\begin{definition}\label{kappa-sequence}
Let $(S_i)_{i\in I}$ be a net of sets of~$\mathcal{A}$; a net $(S_i)_{i\in I}$ is a {\it $\kappa$-net\/} if $\#I \leq \kappa$, where $\#I$~is the cardinality of the set~$I$. In addition, the net $(S_i)_{i\in I}$ is {\it monotone decreasing\/} or {\it monotone increasing\/} if $i \leq j$ implies $S_i \supseteq S_j$ or $S_i \subseteq S_j$, respectively, for any $i, j \in I$.
%\end{definition}

%\begin{definition}\label{defcontin}
Let $\mu \colon \mathcal{A} \to \mathbb{R}$ be a set function. We say that $\mu$ is {\it upper $\kappa$-continuous\/} or {\it lower $\kappa$-continuous\/} at $S \in \mathcal{A}$ if for any monotone decreasing or increasing $\kappa$-net $(S_i)_{i\in I}$ from $\mathcal{A}$ with $\bigcap_{i\in I} S_i = S$ or $\bigcup_{i\in I} S_i = S$, respectively, it holds that $\lim_{i \in I} \mu(S_i) = \mu(S)$. The set function $\mu$~is {\it $\kappa$-continuous\/} if it is both upper and lower $\kappa$-continuous at every set $S \in \mathcal{A}$.
%\end{definition}

Next we define the notion of $\kappa$-additivity. Our definition is similar to the one by \citet{SchervishSeidenfeldKadane2017}.

\begin{definition}\label{ka}
A set function $\mu \colon \mathcal{A} \to \mathbb{R}$ is {\it $\kappa$-additive\/} if it is additive and $\kappa$-continuous. Let $\ba^\kappa(\mathcal{A})$ denote the set of $\kappa$-additive set functions over $\mathcal{A}$.
\end{definition}

%A field $\mathcal{A}$ is a $\kappa$-field, if for all $(A_i)_{i \in I} \subseteq \mathcal{A}$, $| I | < \kappa$, we have $\bigcup_{i \in I} A_i \in \mathcal{A}$. The pair $(X,\mathcal{A})$ is a $\kappa$-measurable space if $\mathcal{A}$ is a $\kappa$-field on $X$.

Note that $\ba^\kappa(\mathcal{A})$ is a linear subspace of $\ba(\mathcal{A})$. Futhermore, the following proposition is easy to see.

\begin{proposition}
If the set function\/ $\mu\colon \mathcal{A} \to \mathbb{R}$ is additive, then it is
\begin{itemize}
\item
upper\/ $\kappa$-continuous if and only if it is lower\/ $\kappa$-continuous;
%\item lower\/ $\kappa$-continuous if and only if it is upper $\kappa$-continuos;
\item
$\kappa$-continuous if and only if it is lower $\kappa$-continuos at $\emptyset$;
%\item
%$\aleph_0$-continuous;
\item
$\aleph_0$-continuous if and only if it is\/ $\sigma$-additive.
\end{itemize}
\end{proposition}

%\begin{remark}
%Let $(X,\mathcal{A})$ be such that $| X | < \kappa$ and $\{x\} \in \mathcal{A}$ for each $x \in X$. Then a bounded non-zero set function $\mu$ on $\mathcal{A}$ is $\kappa$-continuous if and only if it is atomic (not atomless).
%\end{remark}

\begin{example}
The Lebesgue measure on $B([0,1])$, the Borel $\sigma$-field of $[0,1]$,
is not $\kappa$-additive for any $\kappa \geq \mathfrak{c}$, where
$\mathfrak{c}$~denotes the cardinality of the real numbers; but it is
$\kappa$-additive for $\kappa = \aleph_0$.
\end{example}

If $\kappa$ is not countable and the field of sets on which the $\kappa$-additive set function is defined is rich enough, then one may ask whether there are enough or just few $\kappa$-additive set functions. Without going into the details we remark that this problem is related to the notion of measurable cardinal \citep{Ulam1930}.

% Regarding our results, the ``worst'' case is that when there are only
% few such set functions.

The next example shows that there are many $\kappa$-additive set functions in the space $\ba^\kappa(\mathcal{A})$
% abounds with $\kappa$-additive set functions
even in the case when the field $\mathcal{A}$~is large; that is, the theory is not trivial nor vacuous.

\begin{example}\label{pl0}
Let $X$~be an arbitrary set such that $\#X = \kappa \geq \aleph_0$. Consider $\mathcal{P}(X)$, the power set of~$X$. It is clear that the Dirac measures on $\mathcal{P}(X)$ are $\kappa$-additive. Let
$$
\eqarraywithdef {\hfil $\displaystyle #$& #\hfil \cr
{\it \Delta} = \biggl\{\, \sum_{n=1}^\infty \alpha_n \delta_n : {}
	& $(\alpha_n)_{n=1}^\infty \in \ell^1,$\cr
\noalign {\kern-5\jot}
	& $\hphantom{(}%)
	   \delta_n$ being Dirac measures on~$\mathcal{P}(X)$
	   for $n = 1$,~$2$, $3$,~\dots $\,\biggr\} \,.$}
$$
It is clear that each $\mu \in {\it \Delta}$ is a $\kappa$-additive set function on $\mathcal{P}(X)$. Notice that $\#{\it \Delta} \geq \#\ca\bigl(\mathbb{N},\mathcal{P}(\mathbb{N})\bigr)$; that is, even in the ``worst'' case, when there does not exist a non-trivial $\{0,1\}$-valued $\kappa$-additive set function on $\mathcal{P}(X)$, which means the cardinal $\kappa$~is not measurable \citep{Ulam1930}, the collection~${\it \Delta}$ of the trivial $\kappa$-additive set functions on $\mathcal{P}(X)$ is at least as large as the collection of the $\sigma$-additive ones on $\mathcal{P}(\mathbb{N})$.
% (notice that each non-zero measure on $\mathcal{P}(\mathbb{N})$ is a
% countable sum of Dirac measures).
In other words, even in the ``worst'' case, the problem of the non-emptiness of the $\kcore$ is at least as complex as the non-emptiness of the $\sigma$-core with player set~$\mathbb{N}$ and all coalitions feasible, the case considered by \citet{Kannai1969,Kannai1992}.
\end{example}

Given a set system~$\mathcal{A}$, the space $\mathbb{R}^{(\mathcal{A})}$~consists of all functions $\lambda \colon \mathcal{A} \to \mathbb{R}$ with a finite support; that is,
\begin{equation*}
\mathbb{R}^{(\mathcal{A})} = \bigl\{\, \lambda \in \mathbb{R}^\mathcal{A} : \#\{\, S \in \mathcal{A} : \lambda_S \neq 0 \,\} < \infty \,\bigr\} \,.
\end{equation*}
Denoting $\lambda(S)$ and the characteristic function of a set $S \in
\mathcal{A}$ by~$\lambda_S$ and~$\chi_S$, respectively, let
$\Lambda(\mathcal{A}) = \{\, \lambda_{S_1} \chi_{S_1} + \cdots +
\lambda_{S_n} \chi_{S_n} : n \in \mathbb{N}, \allowbreak\!\;
\lambda_{S_1}, \ldots, \allowbreak \lambda_{S_n} \!\in \mathbb{R},
\allowbreak\!\; S_1, \ldots, \allowbreak S_n \in \mathcal{A} \,\}$ be
the space of all simple functions on $(N,\mathcal{A})$; that
is,
\begin{equation*}
\Lambda(\mathcal{A}) =
%\makeatletter \mathopen \bBigg@ {2} \{ \makeatother
\biggl \{\,
\sum_{S\in\mathcal{A}} \lambda_S \chi_S : \lambda \in \mathbb{R}^{(\mathcal{A})}
\,\biggr \} \,.
%\makeatletter \mathclose \bBigg@ {2} \} \makeatother \,.
\end{equation*}
%
%Henceforth, in the summations like $\sum_{S\in\mathcal{A}} \lambda_S \chi_S$ we do not indicate $\lambda_S \neq 0$ explicitly to simplify the notation.

We introduce a norm on~$\Lambda(\mathcal{A})$ as follows. For a simple function $f = \lambda_{S_1} \chi_{S_1} + \cdots + \lambda_{S_n} \chi_{S_n} \!\in \Lambda(\mathcal{A})$
% with $S_1, \ldots, \allowbreak S_n \in \mathcal{A}$. Since $\mathcal{A}$~is a field, we can assume w.l.o.g.\ that $S_1$,~\dots,~$S_n$ are pairwise disjoint.
let
\begin{equation*}
\mathopen\|f\mathclose\| = \sup_{x \in N} \bigl|f(x)\bigr| \,.%= \max \bigl\{ \mathopen|\lambda_{S_1}\mathclose|, \ldots, \mathopen|\lambda_{S_n}\mathclose| \bigr\} \,.
\end{equation*}
%
%It is easy to see that if there exist infinitely many pairwise disjoint sets $S_1, S_2, \ldots \in \mathcal{A}$ then $\Lambda(\mathcal{A})$ is not complete, that is, it is not a Banach space.

Then the topological dual $(\Lambda (\mathcal{A}))^\star$ of the vector space $\Lambda(\mathcal{A})$, which is the space of all continuous linear functionals on $\Lambda(\mathcal{A})$, is isometrically isomorphic to $\ba(\mathcal{A})$, the space of all bounded additive set functions on~$\mathcal{A}$ (see e.g.\ \citet{DunfordSchwartz1958}, Theorem IV.5.1, p.~258). For simplicity, we shall identify the space $(\Lambda(\mathcal{A}))^\star$ with $\ba(\mathcal{A})$. Indeed, a set function $\mu \in \ba (\mathcal{A})$ induces a continuous linear functional $\mu' \in (\Lambda(\mathcal{A}))^\star$ on $\Lambda(\mathcal{A})$ as follows:
\begin{equation}\label{funct}
\mu'(f) = \lambda_{S_1} \mu(S_1) + \cdots + \lambda_{S_n} \mu(S_n)
\end{equation}
for any $f = \lambda_{S_1} \chi_{S_1} + \cdots + \lambda_{S_n} \chi_{S_n} \!\in \Lambda(\mathcal{A})$.

%Let $\ba^\kappa(\mathcal{A}) = \bigl\{\, \mu \in \ba(\mathcal{A}) : \mu$~is bounded and $\kappa$-additive$\,\bigr\}$ be the space of all bounded $\kappa$-additive set functions on~$\mathcal{A}$.

% Moved afore Proposition 5
% Note that $\ba^\kappa(\mathcal{A})$ is a linear subspace of $\ba(\mathcal{A})$.
%Assuming $(N,\mathcal{A})$ is a premeasurable space, the next result will be used in the following sections.

\begin{lemma}\label{dual-pair}
It holds that $\bigl(\Lambda(\mathcal{A}), \ba^\kappa(\mathcal{A})\bigr)$ is a dual pair of spaces.
\end{lemma}

\begin{proof}
Let $f \in \Lambda(\mathcal{A})$ be non-zero, whence there is an $x \in N$ such that $f(x) \neq 0$. Then $\delta_x$, the Dirac measure concentrated at point~$x$ on~$\mathcal{A}$, is a $\kappa$-additive set function, and $\delta'_x(f) = f(x) \neq 0$.
\end{proof}

%In the following we turn our attention to the measure extension problem.

%\begin{corollary}\label{fontos}
%By Equation \eqref{funct} and Lemma \ref{dual-pair} every $\mu \in \ba^\kappa (\mathcal{A})$ can be taken as a continuous linear functional on the subspace $\Lambda(\mathcal{A})$ of the locally convex topological vector space $\Lambda(\mathcal{P}(X))$ equipped with the weak topology of the dual pair $\left(\Lambda(\mathcal{P}(X)), \ba^\kappa(\mathcal{P}(X))\right)$. Therefore, we can apply a continuous variant of the Hahn-Banach Theorem (see e.g. 3.6 Theorem on p. 59 in \citet{Rudin1973}), and we get that $\mu$ can be extended as a bounded $\kappa$-additive set function onto $\mathcal{P} (X)$.
%\end{corollary}

\section{The $\kappa$-core and the $\kappa$-balancedness of TU games}\label{main}

Let $\kappa$~be an arbitrary infinite cardinal number as in the previous section. First, we introduce the notion of TU games. Let $N$~be a non-empty set of players, let $\mathcal{A}' \subseteq \mathcal{P}(N)$ be a collection of sets such that $\emptyset,N \in \mathcal{A}'$, and let $\mathcal{A}$~denote the field hull of~$\mathcal{A}'$; that is, the smallest field of sets that contains~$\mathcal{A}'$. Then a {\it TU game\/} (henceforth a {\it game\/}) on~$\mathcal{A}'$ is a set function $v\colon \mathcal{A}' \to \mathbb{R}$ such that $v(\emptyset) = 0$. We denote the class of games on~$\mathcal{A}'$ by~$\mathcal{G}^{\mathcal{A}'}$. If $\mathcal{A}' = \mathcal{A}$, then $v \in \mathcal{G}^{\mathcal{A}'}$ is a game {\it without restricted cooperation}. Otherwise, if $\mathcal{A}'$~is not a field, $v \in \mathcal{G}^{\mathcal{A}'}$ is a game {\it with restricted cooperation}.

%Recalling $\ba^\kappa(\mathcal{A})$ is the space of all bounded $\kappa$-additive set functions on the field~$\mathcal{A}$

In the following subsections we introduce the three notions of core and the three notions of balancedness that we consider in this paper.

\subsection {The core of a TU game}

First, we introduce the notion of additive core of a game, which was considered by \citet{Schmeidler1967}, \citet{Kannai1969,Kannai1992}, \citet{Pinter2011d}, and \citet{BartlPinter2022a}. %Here we extend the notion to the case of games with restricted cooperation.

\begin{definition}
For a game $v \in \mathcal{G}^{\mathcal{A}'}$ its {\it additive core\/} (henceforth $\bacore$) is defined as follows:
\begin{equation*}
\bacore(v) = \bigl\{\, \mu \in \ba(\mathcal{A}) : \mu(N) = v(N) \text{ and } \mu(S) \geq v(S) \text{ for all } S \in \mathcal{A}' \,\bigr\} \,.
\end{equation*}
\end{definition}

We shall also need the notion of $\sigma$-additive core of a game.

\begin{definition}\label{cacore}
For a game $v \in \mathcal{G}^{\mathcal{A}'}$ its {\it $\sigma$-additive core\/} (henceforth $\cacore$) is defined as follows:
\begin{equation*}
\cacore(v) = \bigl\{\, \mu \in \ca(\mathcal{A}) : \mu(N) = v(N) \text{ and } \mu(S) \geq v(S) \text{ for all } S \in \mathcal{A}' \,\bigr\} \,.
\end{equation*}
\end{definition}

In general, for an infinite cardinal number~$\kappa$ we introduce the notion of $\kcore$ of a game.

\begin{definition}\label{core}
For a game $v \in \mathcal{G}^{\mathcal{A}'}$ its {\it $\kappa$-core\/} is defined as follows:
\begin{equation*}
\kcore(v) = \bigl\{\, \mu \in \ba^\kappa(\mathcal{A}) : \mu(N) = v(N) \text{ and } \mu(S) \geq v(S) \text{ for all } S \in \mathcal{A}' \,\bigr\} \,.
\end{equation*}
\end{definition}

In words, the $\bacore$, the $\cacore$, and the $\kcore$ consists of bounded additive, bounded $\sigma$-additive, and bounded $\kappa$-additive, respectively, set functions defined on the field hull~$\mathcal{A}$ of the feasible coalitions~$\mathcal{A}'$ that meet the  conditions of efficiency ($\mu(N) = v(N)$) and coalitional rationality ($\mu(S) \geq v(S)$ for all $S \in \mathcal{A}'$). Observe that the $\cacore$ is a special case of the $\kcore$ when $\kappa = \aleph_0$.

Notice that in the finite case all the three notions of $\bacore$, $\cacore$, and $\kcore$ are equivalent with the notion of (ordinary) core.

\subsection{Balancedness of a TU game}

%Allow the player set $N$~to be infinite and let $\kappa = \aleph_0$; that is, $\kappa$~is the smallest infinite cardinal. Then $\ba^\kappa(\mathcal{A}) = \ba(\mathcal{A})$ and the $\kappa$-core is the additive core $\bacore(v) = \bigl\{\, \mu \in \ba(\mathcal{A}) : \mu(N) = v(N)$ and $\mu(S) \geq v(S)$ for all $S \in \mathcal{A}' \,\bigr\}$ \citep{Schmeidler1967,Kannai1969,Pinter2011d}.

In the case of infinite games without restricted cooperation with additive core \citet{Schmeidler1967} defined the notion of balancedness. Here, we generalize his notion to the restricted cooperation case, and call it Schmeidler balancedness.% gave the following balancedness condition, where $\mathbb{R}_+^{(\mathcal{A}')} = \{\, \lambda \in \mathbb{R}^{(\mathcal{A}')} : \lambda_S \geq 0$ for all $S \in \mathcal{A}' \,\}$.

\begin{definition}\label{kiegyensulyozott3}
We say that a game $v \in \mathcal{G}^{\mathcal{A}'}$ is {\it Schmeidler balanced\/} if

\begin{equation}\label{tiszta}
\sup
\biggl\{
	\sum_{S\in\mathcal{A}'} \lambda_S v(S) :
	\sum_{S\in\mathcal{A}'} \lambda_S \chi_S = \chi_N, \,
	\lambda \in \mathbb{R}_+^{(\mathcal{A}')}
\biggr\} \leq v(N) \,.
\end{equation}

\end{definition}

%Henceforth we call~\eqref{tiszta} {\it Schmeidler balancedness}. Moreover, unless we mention otherwise, we assume throughout this subsection that $\mathcal{A}' = \mathcal{A} \subseteq \mathcal{P}(N)$ is a field of sets, that is, the game is without restricted cooperation.

%\begin{remark}
Notice that for finite games the notions of Schmeidler balancedness and (ordinary) balancedness \citep{Bondareva1963,Shapley1967,Faigle1989} coincide, hence Schmeidler balancedness is an extension of (ordinary) balancedness.
%\end{remark}

Recall that $Y^* = \ba^\kappa(\mathcal{A})$ is a linear subspace of $Y^\star = \ba(\mathcal{A})$, which can be identified with the topological dual of the normed linear space $Y = \Lambda(\mathcal{A})$. In the next two definitions, where we introduce two new notions of balancedness, we consider the weak topology
% by the dual pair $\bigl(\Lambda(\mathcal{A}), \ba^\kappa(\mathcal{A}) \bigr)$
on $Y = \Lambda(\mathcal{A})$ with respect to $Y^* = \ba^\kappa(\mathcal{A})$ (see Lemma~\ref{dual-pair}).

First, for a game $v \in \mathcal{G}^{\mathcal{A}'}$ consider the convex cone
\begin{equation}\label{cone-K_v^+}
K_v^+ = \Biggl\{ \biggl(
	\sum_{S\in\mathcal{A}'} \lambda_S \chi_S,
	\sum_{S\in\mathcal{A}'} \lambda_S v(S)
\biggr) : \lambda \in \mathbb{R}_+^{(\mathcal{A}')} \Biggr\} \,.
% \eqarraywithdef {
% \hfil $#$& ${} = \hfil #\mathopen{}$& $\hfil #{}$&
% $\displaystyle \hfil#\hfil$& $\displaystyle ,\hfil#\hfil$& ${}#\hfil $&
% $\mathclose{} : \lambda \in \mathbb{R}_+^{(\mathcal{A}')} #\hfil$&
% $#$\hfil \cr
% K_v^+ &
% \Bigl\{\, & \bigl( & A(\lambda) & c(\lambda) & \bigr) & \,\Bigr\}& \,. }
% & \Biggl \{ &
%   \biggl  ( &
%\sum_{S\in\mathcal{A}'} \lambda_S \chi_S &
%\sum_{S\in\mathcal{A}'} \lambda_S v(S)   &
%   \biggr  ) &
%   \Biggr \} &
\end{equation}
%
%we introduce the notion of Schmeidler $\kappa$-balancedness.

% Recall that $A(\lambda) = \sum_{S\in\mathcal{A}'} \lambda_S \chi_S$, \
% $c(\lambda) = \sum_{S\in\mathcal{A}'} \lambda_S v(S)$, \
% $\mathbb{R}_*^{(\mathcal{A}')} = \bigl\{\, \lambda \in
% \mathbb{R}^{(\mathcal{A}')} : \lambda_S \geq 0$ for all $S \in
% \mathcal{A}' \setminus \{N\} \,\bigr\}$, and

\begin{definition}\label{kiegyensulyozott2}
We say that a game $v \in \mathcal{G}^{\mathcal{A}'}$ is {\it Schmeidler $\kappa$-balanced\/} if
$$
z \leq v(N)
$$
for all $z \in \mathbb{R}$ such that $(\chi_N,z) \in \,\overline{\!K_v^+}$, where $\overline{\!K_v^+}$ is the closure of~$K_v^+$.
\end{definition}

%\begin{remark}\label{rem1}
% Notice that the notion of Schmeidler $\kappa$-balancedness is very
% closely related to the notion of supervalue introduced in
% Definition~\ref{supervalue}. The cone $K_v^+$~is precisely the
% set~$D$, if $A(\lambda) = \sum_{S\in\mathcal{A}'} \lambda_S \chi_S$,
% $c(\lambda) = \sum_{S\in\mathcal{A}'} \lambda_S v(S)$ and $P =
% \mathbb{R}_+^{(\mathcal{A}')}$ in Definition~\ref{supervalue}.  Then
% the game is Schmeidler $\kappa$-balanced if and only if the supervalue
% of the related primal problem~$(\mathrm{P_{LP}})$ is not greater
% than~$v(N)$.
%\end{remark}

%\begin{remark}
Observe that Schmeidler $\kappa$-balancedness implies Schmeidler balancedness, which implies (ordinary) balancedness.
%\end{remark}

Lastly, for a game $v \in \mathcal{G}^{\mathcal{A}'}$ let
$$
\mathbb{R}_\ast^{(\mathcal{A}')} = \bigl\{\, \lambda \in \mathbb{R}^{(\mathcal{A}')} : \lambda_S \geq 0 \text{ for all } S \in \mathcal{A}' \setminus \{N\} \,\bigr\}
$$
and consider the convex cone
\begin{equation}\label{cone-K_v}
K_v = \Biggl\{ \biggl(
	\sum_{S\in\mathcal{A}'} \lambda_S \chi_S,
	\sum_{S\in\mathcal{A}'} \lambda_S v(S)
\biggr) : \lambda \in \mathbb{R}_\ast^{(\mathcal{A}')} \Biggr\} \,.
% \eqarraywithdef {
% \hfil $#$& ${} = \hfil #\mathopen{}$& $\hfil #{}$&
% $\displaystyle \hfil#\hfil$& $\displaystyle ,\hfil#\hfil$& ${}#\hfil $&
% $\mathclose{} : \lambda \in \mathbb{R}_*^{(\mathcal{A}')} #\hfil$&
% $#$\hfil \cr
% K_v & \Bigl\{\, & \bigl( & A(\lambda) & c(\lambda) & \bigr) & \,\Bigr\}& \,, }
%  & \Biggl \{ &
%    \biggl  ( &
% \sum_{S\in\mathcal{A}'} \lambda_S \chi_S &
% \sum_{S\in\mathcal{A}'} \lambda_S v(S)   &
%    \biggr  ) &
%    \Biggr \} &
\end{equation}
%
%we introduce the notion of $\kappa$-balancedness of the game~$v$.

\begin{definition}\label{kiegyensulyozott}
A game $v \in \mathcal{G}^{\mathcal{A}'}$ is {\it $\kappa$-balanced\/} if
\begin{equation*}
z \leq v(N)
\end{equation*}
for all $z \in \mathbb{R}$ such that $(\chi_N,z) \in \,\overline{\!K_v}$, where $\overline{\!K_v}$ is the closure of~$K_v$.
%, where the closure is understood in the product of the weak topology with respect to $\ba^\kappa(\mathcal{A})$ and the Euclidean topology on the space $\Lambda(\mathcal{A}) \times \mathbb{R}$.
\end{definition}

\begin{remark}\label{rem11}
The notion of $\kappa$-balancedness and Schmeidler $\kappa$-balancedness is very closely related to the notion of supervalue introduced in Definition~\ref{supervalue}. The cone $K_v$~or~$K_v^+$ is precisely the set~$D$ if $A(\lambda) = \sum_{S\in\mathcal{A}'} \lambda_S \chi_S$ and $c(\lambda) = \sum_{S\in\mathcal{A}'} \lambda_S v(S)$ with $P = \mathbb{R}_+^{(\mathcal{A}')}$ or $P = \mathbb{R}_*^{(\mathcal{A}')}$, respectively, in Definition~\ref{supervalue}. Then the game is $\kappa$-balanced or Schmeidler $\kappa$-balanced, respectively, if and only if the supervalue of the related primal problem~$(\mathrm{P_{LP}})$ is not greater than~$v(N)$.
\end{remark}

Notice that the notion of $\kappa$-balancedness is a ``double" extension of Schmeidler balancedness. First, we do not take the balancing weight system alone, but we take nets of balancing weight systems. Second, we let the weight of the grand coalition be sign unrestricted. It is worth noticing that the notion of $\kappa$-balancedness applies its full strength when in a net of balancing weight systems the net of the weights of the grand coalition is not bounded below (see Lemma~\ref{egyenloseg} below).

The insight why we need the ``double" extension is the following: As we shall see, the proof of our generalized Bondareva-Shapley theorem is based on the strong duality theorem for infinite LPs (Proposition~\ref{duality}), which is based on separation of a closed convex set from a point (not in the set). Therefore, we need to take the weak closure of a convex set in our proof. This is why we use the nets of balancing set systems.

Regarding that the weight of the grand coalition is sign unrestricted, notice that the linear combinations of Dirac measures are $\kappa$-additve for any~$\kappa$, moreover, it is easy to see that the linear space spanned by the Dirac measures is weak* dense in the set of bounded additive set functions. Hence, by the results of \citet{Schmeidler1967}, \citet{Kannai1969,Kannai1992}, and \citet{BartlPinter2022a},
% and Theorem~\ref{megszamjatek-Schmeidler}
we have a necessary and sufficient condition for the non-emptiness of the ``approximate" $\kcore$ for any~$\kappa$ for free: Schmeidler balancedness. However, we analyze the non-emptiness of the (exact) $\kcore$ for any~$\kappa$. Therefore,
% we apply the well-known ``trick" of duality for LP problems:
we set the appropriate variable (the weight of the grand coalition) in the primal problem be sign unrestricted, by which we get equality in the related constraint in the dual problem (the total mass of an allocation must exactly be the value of the grand coalition), hence we will have a necessary and sufficient condition for the non-emptiness of the $\kcore$ for any~$\kappa$: $\kappa$-balancedness.

Between Schmeidler balancedness and $\kappa$-balancedness, there lies the ``double" extension of the former one, Schmeidler $\kappa$-balancedness, where only the first step is taken: we take nets of balancing weight systems. Even though we shall see later that Schmeidler $\kappa$-balancedness does not lead to new characterization results, it provides deeper understanding of the problem.

%One further remark related to the discussion above, if $\kappa$~is not countable and the set of feasible coalitions is rich enough, then one could ask whether there are enough $\kappa$-additive set functions to get our characterization result (Theorem~\ref{megszamjatek1}). Without going into details, we remark that this problem also relates to the notion of measurable cardinal \citep{Ulam1930}. Put differently, our characterization result should somehow mention how rich the set of feasible solutions is. By that the linear space spanned by the Dirac measures is weak* dense in the set of bounded additive set functions, however, for any cardinal $\kappa$ it is enough to consider the linear space of Dirac measures, which are $\kappa$-additive, as the possible solution set of the problem. Therefore, our results are not meager.

%However, a natural question can rise whether instead of the "double" jump a "single" jump is not enough. The non-trivial "single" jump leads to the following notion.

%Before we consider the finite case ($|N| < \infty$) consider the following lemma, where

Since Schmeidler $\kappa$-balancedness is the same as $\kappa$-balancedness except that~$K_v$ in Definition~\ref{kiegyensulyozott2} is replaced by~$K_v^+$ in Definition~\ref{kiegyensulyozott}, by $K_v^+ \subseteq K_v$, it is clear that $\kappa$-balancedness implies Schmeidler $\kappa$-balancedness. Furthermore, Schmeidler $\kappa$-balancedness and $\kappa$-balancedness are related by the following lemma.
%, where
%$$
%\mathbb{R}_+^{(\mathcal{A}')} = \bigl\{\, \lambda \in \mathbb{R}^{(\mathcal{A}')} : \lambda_S \geq 0 \text{ for all } S \in \mathcal{A}' \,\bigr\} \,,
%$$
%uses the notation introduced in the proof of Theorem~\ref{megszamjatek1}.

\begin{lemma}\label{egyenloseg}
For a game\/ $v \in \mathcal{G}^{\mathcal{A}'}$ it holds
$$
\sup_{\substack{
	(\lambda^i)_{i\in I} \subseteq \mathbb{R}_+^{(\mathcal{A}')} \\
	A(\lambda^i) \buildrel w \over \longrightarrow \chi_N \\
	c(\lambda^i) \longrightarrow z
}} z \,\, \leq \, v(N)
\qquad\quad
\text{if and only if}
\quad\qquad
\sup_{\substack{
	(\lambda^j)_{j\in J} \subseteq \mathbb{R}_*^{(\mathcal{A}')} \\
	A(\lambda^j) \buildrel w \over \longrightarrow \chi_N \\
	c(\lambda^j) \longrightarrow z \\
	\hidewidth
	\liminf \lambda^j_N > -\infty
	\hidewidth
}} z \,\, \leq \, v(N) \,.
$$
where\/ $A (\lambda) = \sum_{S \in \mathcal{A}'} \lambda_S \chi_S$ and\/ $c (\lambda) = \sum_{S \in \mathcal{A}'} \lambda_S v(S)$ for any\/ $\lambda \in \mathbb{R}_+^{\mathcal{A}'}$.
%
%\/ $\liminf \lambda^j_N = \sup \{\, L \in \mathbb{R} : \exists j_0 \in J \allowbreak\,\, \forall j \in J\colon \allowbreak\, j \geq j_0 \,\Rightarrow\mskip.5\thinmuskip \lambda^j_N \geq L \,\}$.
\end{lemma}

\begin{proof}
The ``if'' part is obvious.  Given a net $(\lambda^i)_{i\in I} \subseteq
\mathbb{R}_+^{(\mathcal{A}')}$, consider the same net $(\lambda^j)_{j\in
J} = (\lambda^i)_{i\in I} \subseteq \mathbb{R}_*^{(\mathcal{A}')}$.
Notice that $\liminf \lambda^j_N \geq 0$.

We prove the ``only if'' part indirectly.  Suppose the right-hand side
does not hold.  Then there exists a net $(\lambda^j)_{j \in J} \subseteq
\mathbb{R}_*^{(\mathcal{A}')}$ such that $\liminf \lambda^j_N = L >
-\infty$ and $A(\lambda^j) \buildrel w \over \longrightarrow \chi_N$
with $c(\lambda^j) \longrightarrow z > v(N)$.

If $L > 0$, then there exists a $j_0 \in J$ such that $j \geq j_0$
implies $\lambda^j_N \geq 0$.  Consider the index set $I = \{\, j \in J
: j \geq j_0 \,\}$ and the net $(\lambda^i)_{i\in I} \subseteq
\mathbb{R}_+^{(\mathcal{A}')}$, which satisfies $A(\lambda^i) \buildrel
w \over \longrightarrow \chi_N$ and $c(\lambda^j) \longrightarrow z >
v(N)$.

Assume $L \leq 0$.  There exists a subnet $(\lambda^{j_i})_{i\in I}$ of
$(\lambda^j)_{j\in J}$ such that $\lambda^{j_i}_N \longrightarrow L$.
Define the net $(\bar\lambda^i)_{i\in I}$ as follows: for any $i \in I$
and for any $S \in \mathcal{A}'$ let
\begin{equation*}
\bar\lambda^i_S =
	\begin{cases}
\hphantom{-} 0                         & \text{if $S = N$,}\\
\hphantom{-} \lambda^{j_i}_S / (1 - L) & \text{otherwise.}\\
	\end{cases}
\end{equation*}
Then
$$
\eqarraywithdef {\hfil $\displaystyle #$& $\displaystyle {}#$\hfil \cr
A(\bar\lambda^i) &=
\sum_{\substack{ S\in\mathcal{A}' \\ S\neq N }} \frac{\lambda^{j_i}_S \chi_S}{1 - L} =
\frac{A(\lambda^{j_i}) - \lambda^{j_i}_N \chi_N}{1 - L} \cr &\buildrel w \over \longrightarrow
\frac{\chi_N - L \chi_N}{1 - L} = \chi_N \,, }
$$
and
$$
\eqarraywithdef {\hfil $\displaystyle #$& $\displaystyle {}#$\hfil \cr
c(\bar\lambda^i) &=
\sum_{\substack{ S\in\mathcal{A}' \\ S\neq N }} \frac{\lambda^{j_i}_S v(S)}{1 - L} =
\frac{c(\lambda^{j_i}) - \lambda^{j_i}_N v(N)}{1 - L} \cr &\longrightarrow
\frac{z - L v(N)}{1 - L} >
\frac{v(N) - L v(N)}{1 - L} = v(N) \,. }
$$

It follows that the left-hand side does not hold in either case, which
concludes the proof.
\end{proof}

%\begin{remark}
%Notice that Schmeidler $\kappa$-balancedness is the same as $\kappa$-bal\-ancedness except that~$K_v$, see~\eqref{cone-K_v}, is replaced by~$K_v^+$ in Definition~\ref{kiegyensulyozott}. Although $\kappa$-balancedness implies Schmeidler $\kappa$-balancedness since $K_v^+ \subseteq K_v$, Schmeidler $\kappa$-balancedness and $\kappa$-balancedness are related by Lemma~\ref{egyenloseg}. It says that a game $v$~is Schmeidler $\kappa$-balanced if and only if it holds
%$$
%\sup_{\substack{
%	(\lambda^j)_{j\in J} \subseteq \mathbb{R}_*^{(\mathcal{A}')} \\
%	A(\lambda^j) \buildrel w \over \longrightarrow \chi_N \\
%	c(\lambda^j) \longrightarrow z \\
%	\hidewidth
%	\liminf \lambda^j_N > -\infty
%	\hidewidth
%}} z \,\, \leq \, v(N) \,.
%$$
%\end{remark}

\section{The main result}\label{main2}

The next result is our generalized Bondareva-Shapley Theorem.

\begin{theorem}\label{megszamjatek1}
For any game\/ $v \in \mathcal{G}^{\mathcal{A}'}$ it holds that\/ $\kcore(v) \neq \emptyset$ if and only if the game is\/ $\kappa$-balanced.
\end{theorem}

\begin{proof}
Put $X = \mathbb{R}^{(\mathcal{A}')}$, \ $P = \mathbb{R}_\ast^{(\mathcal{A}')}$, \ $Y = \Lambda(\mathcal{A})$, and $Y^* = \ba^\kappa(\mathcal{A})$, moreover define the mapping $A\colon \mathbb{R}^{(\mathcal{A}')} \to \Lambda(\mathcal{A})$ by $A(\lambda) = \sum_{S\in\mathcal{A}'} \lambda_S \chi_S$, let $b = \chi_N$, and define the functional $c\colon \mathbb{R}^{(\mathcal{A}')} \to \mathbb{R}$ by $c(\lambda) = \sum_{S\in\mathcal{A}'} \lambda_S v(S)$. Now, consider the programs $(\mathrm{P_{LP}})$~and~$(\mathrm{D_{LP}})$ of~\eqref{LP}.

Notice that program $(\mathrm{P_{LP}})$~is superconsistent and its supervalue is at least~$v(N)$. (Consider that $\bigl(A(\lambda), c(\lambda)\bigr) \in K_v \subseteq \,\overline{\!K_v}$ for $\lambda \in \mathbb{R}^{(\mathcal{A}')}$ with $\lambda_N = 1$ and $\lambda_S = 0$ for $S \neq N$.) Then the game is $\kappa$-balanced (Definition~\ref{kiegyensulyozott}) if and only if the supervalue of~$(\mathrm{P_{LP}})$ is finite and not greater than~$v(N)$ (Remark~\ref{rem11}).

Moreover, observe that a set function $\mu \in \ba^\kappa(\mathcal{A})$ is feasible for~$(\mathrm{D_{LP}})$ if and only if $\mu(S) \geq v(S)$ for all $S \in \mathcal{A}'$ and $\mu(N) = v(N)$. Thus program $(\mathrm{D_{LP}})$~is equivalent to finding an element of $\kcore(v)$, and its value is~$v(N)$ if it is consistent, and its value is~$+\infty$ otherwise.

Therefore by Proposition~\ref{duality} the game has a non-empty $\kcore$ (program $(\mathrm{D_{LP}})$~is consistent) if and only if it is $\kappa$-balanced (the supervalue of program $(\mathrm{P_{LP}})$ is not greater than~$v(N)$).
\end{proof}

%\subsection{The finite case}\label{finite}

%We are now turning our attention to the finite player set case.
If the player set $N$~is finite, then so is $\mathcal{A}' \subseteq \mathcal{P}(N)$, whence the cone $K_v$~is closed. Then by Lemma~\ref{egyenloseg} $\kappa$-balancedness reduces to Schmeidler balancedness, which is (ordinary) balancedness
%can be written as: a game $v \in \mathcal{G}^{\mathcal{A}'}$ is $\kappa$-balanced if
%
%\begin{equation}\label{veges}
%\max
%\biggl\{
%\sum_{S\in\mathcal{A}'} \lambda_S v(S) :
%\sum_{S\in\mathcal{A}'} \lambda_S \chi_S = \chi_N, \, \lambda \in \mathbb{R}_+^{\mathcal{A}'}
%\biggr\} \leq v(N) \,.
%\end{equation}
%
%Notice that the $\kappa$-core is the classical core and \eqref{veges}~is the classical definition of balancedness
\citep{Bondareva1963,Shapley1967,Faigle1989},
%that is, the definition of $\kappa$-balancedness is a natural generalization of the balancedness of finite games (without and with restricted cooperation).
and the $\kcore$ is the (ordinary) core in the finite case.
%Hence we have obtained the following corollary of Theorem~\ref{megszamjatek1}.
We thus obtain the classical Bondareva-Shapley Theorem as a corollary of Theorem~\ref{megszamjatek1}:

\begin{corollary}[Bondareva-Shapley Theorem]\label{cor1}
If\/ $N$~is finite, then the core of a game with or without restricted cooperation is non-empty if and only if the game is balanced.
\end{corollary}

Regarding Theorem~\ref{megszamjatek1}, it is worth mentioning that while \citet{Bondareva1963} applied the strong duality theorem to prove the Bondareva-Shapley Theorem, \citet{Shapley1967} used a different approach. We do not go into the details, but we remark that the common point in both approaches is the application of a separating hyperplane theorem. In other words, both \citeauthor{Bondareva1963}'s and \citeauthor{Shapley1967}'s approaches are based on the same separating hyperplane theorem, practically their result is a direct corollary of that. Here we use the strong duality theorem for infinite LPs \citep[Proposition~\ref{duality},][]{AndersonNash1987}, which is also a direct corollary of the same separating hyperplane theorem.

\subsection{The $\sigma$-additive case}%label{sigma}

In this subsection let $\kappa = \aleph_0$. Then $\ba^\kappa(\mathcal{A}) = \ca(\mathcal{A})$, the space of all bounded countably additive set functions on~$\mathcal{A}$. Given a game $v \in \mathcal{G}^{\mathcal{A}'}$, its $\kcore$ is the $\sigma$-additive core $\cacore(v)$ introduced by Definition~\ref{cacore}.
% = \bigl\{\, \mu \in \ca(\mathcal{A}) \colon \mu(N) = v(N)$ and $\mu(S)
% \geq v(S)$ for all $S \in \mathcal{A}' \,\bigr\}$.

%In Example~\ref{ex1} we presented a Schmeidler $\kappa$-balanced unbounded below game without restricted cooperation having its $\bacore$ empty.
In the next example we demonstrate that there exists a Schmeidler $\kappa$-balanced ($\aleph_0$-balanced) non-negative game without restricted cooperation having its $\cacore$ empty.% for $\kappa=\aleph_1$.

\begin{example}\label{pl2}
Let the player set $N = \mathbb{N}$, the system of coalitions
$\mathcal{A} = \mathcal{P}(\mathbb{N})$, and the game $v$~be defined as
follows: for any $S \in \mathcal{A}$ let
\begin{equation*}
v(S) =
	\begin{cases}
\hphantom{-} 1 & \text{if $\#(N \setminus S) \leq 1$,}\\
\hphantom{-} 0 & \text{otherwise.}\\
	\end{cases}
\end{equation*}
We show that $\cacore(v) = \emptyset$.  If $\mu \in \cacore(v)$, then
$\mu\bigl(N \setminus \{n\}\bigr) \geq v\bigl(N \setminus \{n\}\bigr) =
1$ and $v(N) = 1$, whence $\mu\bigl(\{n\}\bigr) \leq 0$.  So $0 \geq
\sum_{n=1}^\infty \mu\bigl(\{n\}\bigr) = \mu(N) = v(N) = 1$, a
contradiction.

We now show that, if $(\chi_N,z) \in \,\overline{\!K_v^+}$,
see~\eqref{cone-K_v^+},
% where the closure of the cone is understood in the product of the weak
% topology with respect to~$\ca(\mathcal{A})$ and the Euclidean topology
% on the space $\Lambda(\mathcal{A}) \times \mathbb{R}$,
then $z \leq 1 = v(N)$.  We have $(\chi_N,z) \in
\,\overline{\!K_v^+}$ if and only if each neighborhood of the point
$(\chi_N,z)$ intersects the cone~$K_v^+$.  In particular, if $(\chi_N,z)
\in \,\overline{\!K_v^+}$, then for any natural number~$m$ and for any
$\varepsilon > 0$ there exists a point $(f,t) \in K_v^+$ such that
$f$~belongs to the weak neighborhood
\begin{equation*}
\bigl\{\, f \in \Lambda(\mathcal{A}) :
	\bigl|\delta'_i(f) - 1\bigr| < \varepsilon
	\enskip \text{for} \enskip i = 1,\, \ldots,\, m
\,\bigr\} \,,
\end{equation*}
where $\delta'_i$~is the continuous linear functional induced by the
Dirac measure $\delta_i$ concentrated at~$i$, see~\eqref{funct}, and
$t$~belongs to the neighborhood $\bigl\{\, t \in \mathbb{R} :
\mathopen|t - z\mathclose| < \varepsilon \,\bigr\}$.  Hence, we have a
natural number~$n$, some distinct sets $S_0, S_1, \ldots, \allowbreak
S_n \in \mathcal{A}$, and some non-negative
$\lambda_{S_0}$,~$\lambda_{S_1}$,~\dots,~$\lambda_{S_n}$ such that $f =
\lambda_{S_0} \chi_{S_0} + \lambda_{S_1} \chi_{S_1} + \cdots +
\lambda_{S_n} \chi_{S_n}$ and
\begin{equation}\label{pl2aux1}
\biggl|
\sum_{\substack{ j=0 \\ \hidewidth S_j\ni i \, \hidewidth}}^n \lambda_{S_j} - 1
\biggr| < \varepsilon \qquad \text{for} \quad i = 1,\, \ldots,\, m
\end{equation}
with
\begin{equation}\label{pl2aux2}
\biggl|
\sum_{j=0}^n \lambda_{S_j} v(S_j) - z
\biggr| =
\biggl|
\sum_{\substack{ j=0 \\ \hidewidth \#(N\setminus S_j)\leq 1\, \hidewidth}}^n \lambda_{S_j} - z
\biggr| < \varepsilon \,.
\end{equation}
We can assume w.l.o.g.\ that $S_0 = N$, as well as $\#(N \setminus S_j) = 1$
for $j = 1$,~\dots,~$n_1$ and $\#(N \setminus S_j) > 1$ for $j = n_1 +
1$,~\dots,~$n$, where $n_1 \leq n$.

Everything is clear if there exists an $i \in \{1, \ldots, m\}$ such
that $i \in \bigcap_{j=1}^{n_1} S_j$.  Then by~\eqref{pl2aux1}
$$
\sum_{\substack{ j = 0 \\ \hidewidth \#(N\setminus S_j) \leq 1 \, \hidewidth}}^n \lambda_{S_j} =
\sum_{j=0}^{n_1} \lambda_{S_j} \leq
\sum_{\substack{ j = 0 \\ \hidewidth S_j \ni i \, \hidewidth}}^n \lambda_{S_j}
< 1 + \varepsilon \,,
$$
whence $z < 1 + 2\varepsilon$ by~\eqref{pl2aux2}.

In the other case we have $m \leq n_1$ and, because the sets
$S_0$,~$S_1$,~\dots,~$S_n$ are pairwise distinct, for $i =
1$,~\dots,~$m$ we can assume w.l.o.g.\ that $S_i = N \setminus \{i\}$.
By~\eqref{pl2aux1}
$$
\sum_{j=0}^{n_1} \lambda_{S_j} - \lambda_{S_i} =
\sum_{\substack{ j=0 \\ \hidewidth j\neq i \hidewidth}}^{n_1} \lambda_{S_j} \leq
\sum_{\substack{ j=0 \\ \hidewidth S_j\ni i \, \hidewidth}}^n \lambda_{S_j} <
1 + \varepsilon
\qquad \text{for} \quad i = 1,\, \ldots,\, m \,.
$$
Summing up, we get $m\sum_{j=0}^{n_1} \lambda_{S_j} - \sum_{i=1}^m
\lambda_{S_i} < m + m\varepsilon$, whence $m\sum_{j=0}^{n_1}
\lambda_{S_j} - \sum_{j=0}^{n_1} \lambda_{S_i} < m + m\varepsilon$.  It
then follows
$$
\sum_{\substack{ j=0 \\ \hidewidth \#(N\setminus S_j)\leq 1 \, \hidewidth}}^n \lambda_{S_j} =
\sum_{j=0}^{n_1} \lambda_{S_j} <
\frac{ m }{ m - 1} (1 + \varepsilon) \,.
$$
Taking~\eqref{pl2aux2} into account, we obtain
\begin{equation}\label{pl2fin}
z < \frac{ m }{ m - 1} (1 + \varepsilon) + \varepsilon \,.
\end{equation}

Since $1 + 2 \varepsilon < (1 + \varepsilon)m / (m - 1) + \varepsilon$, inequality \eqref{pl2fin}~holds in both cases.  By that $m \geq
2$ and $\varepsilon > 0$ can be arbitrary, we conclude that $z \leq 1$.
\end{example}

\begin{remark}
Consider the game $v$~from Example~\ref{pl2}.  Since $\cacore(v) =
\emptyset$, the game is not $\kappa$-balanced ($\aleph_0$-balanced).  To see this, consider
the sequence $(\lambda^i)_{i=1}^\infty$, with $\lambda^i \in
\mathbb{R}_*^{(\mathcal{A})}$, defined as follows: for any $i \in
\mathbb{N}$ and for any $S \in \mathcal{A}$ let
\begin{equation*}
\lambda^i_S =
	\begin{cases}
-(i - 2)           & \text{if $S = N$,}\\
\hphantom{-} 1     & \text{if $S = N \setminus \{n\}$ for $n = 1$,~\dots,~$i$,}\\
\hphantom{-} 0     & \text{otherwise.}\\
	\end{cases}
\end{equation*}
Then $\sum_{S\in\mathcal{A}} \lambda^i_S \chi_S = 2\chi_N -
\chi_{\{1,\ldots,i\}} \buildrel w \over \longrightarrow \chi_N$, where
the weak convergence in the space~$\Lambda(\mathcal{A})$ is with respect
to~$\ca(\mathcal{A})$, and $\sum_{S \in \mathcal{A}} \lambda^i_S v (S) =
2 > 1 = v(N)$.

Notice again that the sequence $(\lambda^i_N)_{i=1}^\infty = (2 -
i)_{i=1}^\infty$ is unbounded below.  If $(\lambda^i_N)_{i=1}^\infty$
were bounded below, then by Lemma~\ref{egyenloseg} we would get a
contradiction with Example~\ref{pl2}.
\end{remark}

Example~\ref{pl2} demonstrates that it is not sufficient to use $\mathbb{R}_+^{(\mathcal{A})}$ and~$K_v^+$ in the definition of $\kappa$-balancedness; that is, Schmeidler $\kappa$-balancedness is unable to reveal that the $\cacore$ is empty even for non-negative games without restricted cooperation. %The net $(\lambda^i_N)_{i\in I}$ must be let unbounded below.

\begin{remark}\label{sigmaremark}
Reconsidering Schmeidler balancedness %and Theorem~\ref{megszamjatek-Schmeidler}
for the additive case, it is somehow tempting to ask whether the following ``$\sigma$-ex\-tension'' of condition~\eqref{tiszta} could lead to a similar result in the $\sigma$-additive case too:
\begin{equation}\label{sigma-balancedness}
\sup
\biggl\{\,
\sum_{S\in\mathcal{A}'} \lambda_S v(S) :
\sum_{S\in\mathcal{A}'} \lambda_S \chi_S = \chi_N, \, \lambda \in \mathbb{R}_+^{[\mathcal{A}']}
\biggr\} \leq v(N) \,,
\end{equation}
where $\mathbb{R}^{[\mathcal{A}']} = \bigl\{\, \lambda \in \mathbb{R}^{\mathcal{A}'} : \#\{\, S \in \mathcal{A}' : \lambda_S \neq 0 \,\} \leq \aleph_0 \,\bigr\}$ and $\mathbb{R}_+^{[\mathcal{A}']} = \bigl\{\, \lambda \in \mathbb{R}^{[\mathcal{A}']} : \lambda_S \geq 0$ for all $S \in \mathcal{A}' \,\bigr\}$. Moreover, the convergence of the sum $\sum_{S\in\mathcal{A}'} \lambda_S \chi_S$ is understood pointwise. In this case it is equivalent to say that the convergence is weak in the space~$\Lambda(\mathcal{A})$ with respect to $\ca(\mathcal{A})$.
%
% The value of the sum $\sum_{S\in\mathcal{A}} \lambda_S v(S)$ is well defined if it converges absolutely or if it diverges to~$+\infty$ or~$-\infty$.
%
If the sum $\sum_{S\in\mathcal{A}'} \lambda_S v(S)$ is convergent, but not absolutely convergent, then we put $\sum_{S\in\mathcal{A}'} \lambda_S v(S) :=+\infty$.

% Furthermore, notice that if $(\lambda^i)_{i \in I} \subseteq \mathbb{R}_+^{[\mathcal{A}]}$ is such that $A (\lambda^i) \buildrel w \over \longrightarrow \chi_N$ and $c (\lambda^i) \longrightarrow z$, then by the fact that for each $i \in I$ there exists $(\lambda^{i n})_{n=1} ^\infty \subseteq \mathbb{R}^{(\mathcal{A})}_+$ such that $A (\lambda^{i n}) \buildrel w \over \longrightarrow A (\lambda^i)$ and $c(\lambda^{i n}) \longrightarrow c(\lambda^i)$, it holds that there exists $(\lambda^j)_{j \in J} \subseteq \mathbb{R}_+^{(\mathcal{A})}$ such that $A (\lambda^j) \buildrel w \over \longrightarrow \chi_N$ and $c (\lambda^j) \longrightarrow z$, where $A (\lambda^i) = \sum_{S \in \mathcal{A}} \lambda^i_S \chi_S$ and $c (\lambda^i) = \sum_{S\in\mathcal{A}} \lambda^i_S v(S)$, $i \in I$. In other words, Schmeidler $\kappa$-balancedness covers these kinds of generalizations.

Denoting $A(\lambda) = \sum_{S\in\mathcal{A}'} \lambda_S \chi_S$ and $c(\lambda) = \sum_{S\in\mathcal{A}'} \lambda_S v(S)$, we can also consider the following generalization of~\eqref{sigma-balancedness}. Let $z \leq v(N)$ whenever there exists a net $(\lambda_i)_{i\in I} \subseteq \mathbb{R}_+^{[\mathcal{A}']}$ such that $A (\lambda^i) \buildrel w \over \longrightarrow \chi_N$ and $c (\lambda^i) \longrightarrow z$ where $z$~is finite. Then for each $i \in I$ there exists a sequence $(\lambda^{in})_{n=1}^\infty \subseteq \mathbb{R}^{(\mathcal{A}')}_+$ such that $A(\lambda^{in}) \buildrel w \over \longrightarrow A (\lambda^i)$ and $c(\lambda^{in}) \longrightarrow c(\lambda^i)$. Consequently, there exists a net $(\lambda_j)_{j\in J} \subseteq \mathbb{R}_+^{(\mathcal{A}')}$ such that $A(\lambda^j) \buildrel w \over \longrightarrow \chi_N$ and $c(\lambda^j) \longrightarrow z$. In other words, Schmeidler $\kappa$-balancedness covers such extensions of Schmeidler balancedness (Definition \ref{kiegyensulyozott3}) like~\eqref{sigma-balancedness}.

Moreover, in Example~\ref{pl2} we presented a non-negative Schmeidler $\kappa$-balanced game. Therefore, the presented game is balanced according to~\eqref{sigma-balancedness} too, but the $\cacore$ of the game is empty.
%
%Therefore, the game in Example~\ref{pl2} is balanced according to~\eqref{sigma-balancedness} and its generalization, although its $\cacore$ is empty.
%
%Then condition \eqref{sigma-balancedness}~holds for the game $v$~of Example~\ref{pl2} too, that is, the condition \eqref{sigma-balancedness}~is also unable to reveal that the $\cacore$ of the game is empty.
%
%To see that condition \eqref{sigma-balancedness}~holds for the game~$v$, assume that there is a $\bar\lambda \in \mathbb{R}_+^{[\mathcal{A}]}$ such that $\sum_{S\in\mathcal{A}} \lambda_S \chi_S = \chi_N$ and $\sum_{S\in\mathcal{A}} \lambda_S v(S) = z > v(N)$.  Let $S_1, S_2, \allowbreak S_3, \ldots \in \mathcal{A}$ be distinct sets such that $\{S_1, S_2, \allowbreak S_3, \ldots\} = \{\, S \in \mathcal{A} : \bar\lambda_S \neq 0 \,\}$.  Then consider the sequence $(\lambda^i)_{i=1}^\infty \subseteq \mathbb{R}_+^{(\mathcal{A})}$ defined as follows: for any $i \in \mathbb{N}$ and for any $S \in \mathcal{A}$ let $\lambda^i_S = \bar\lambda_S$ if $S \in \{S_1, \ldots, S_i\}$, and let $\lambda^i_S = 0$ otherwise.
%
%Then $\sum_{S\in\mathcal{A}} \lambda_S \chi_S \longrightarrow \chi_N(x)$ pointwise, hence $\sum_{S\in\mathcal{A}} \lambda^i_S \chi_S \buildrel w \over \longrightarrow \chi_N(x)$ where the weak convergence is with respect to $\ca(\mathcal{A})$.  And we also have $\sum_{S\in\mathcal{A}} \lambda^i_S v(S) \longrightarrow z$.  We know by Example~\ref{pl2} that, if $(\chi_N,z) \in \,\overline{\!K_v^+}$, then $z \leq v(N)$.
\end{remark}

\section{Conclusion}

We have generalized the Bondareva-Shapley Theorem to TU games with and without restricted cooperation, with infinitely many players, and with at least $\sigma$-additive cores:
% . In the general case
we have proved for an arbitrary infinite cardinal~$\kappa$ that the $\kcore$ of a TU game with or without restricted cooperation is not empty if and only if the TU game is $\kappa$-balanced. The main conceptual messages of our results might be that in the proper notion of balancing weight system the weight of the grand coalition is sign unrestricted.%not sign restricted.%, yet unbounded below.

While $\kappa$-balancedness is universally a necessary and sufficient
condition for that the $\kcore$ of a game with or without restricted
cooperation is not empty (Theorem~\ref{megszamjatek1}), we have shown
that Schmeidler $\kappa$-balancedness (which implies Schmeidler
balancedness as well as its $\sigma$-extension, see
Remark~\ref{sigmaremark}) is not suitable for this purpose even in the
case of $\sigma$-additive core of a non-negative game without restricted
cooperation (Example~\ref{pl2}).

Notice that \citet{Kannai1969,Kannai1992} gave another necessary and sufficient condition for that the $\sigma$-additive core of a non-negative game without restricted cooperation is not empty. \citeauthor{Kannai1992}'s result is based on a very different approach and not related directly to our $\aleph_0$-balancedness condition.

%Regarding applications, the non-emptiness of the core is relevant from the viewpoint of decision theory, where a non-additive belief is a monotone TU game, and the so called priors (possible distributions, beliefs) are the elements of the core of the corresponding TU game \citep{Schmeidler1989,GilboaSchmeidler1989,GhirardatoMarinacci2002}. Recently \citet{Cerreia-VioglioGhirardatoMaccheroniMarinacciSiniscalchi2011} showed that under quite general assumptions a non-additive belief expresses ambiguity aversion if and only if its core is not empty. In other words, a necessary and sufficient condition for the non-emptiness of the core is a necessary and sufficient condition of that a non-additive belief expresses ambiguity aversion.

%\bibliographystyle{/home/pmiklos/Dropbox/Munka/Styles/ANOR/spbasic}
%\bibliography{/home/pmiklos/Dropbox/Munka/Hivatkozasok/referencesPMP.bib}

%\bibliographystyle{/Users/miklospinter/Dropbox/Munka/Styles/ANOR/spbasic}
%\bibliography{/Users/miklospinter/Dropbox/Munka/Hivatkozasok/referencesPMP.bib}

\end{document}